\documentclass[a4paper,10pt]{amsart}

\usepackage[english]{babel}

\usepackage[a4paper,top=2cm,bottom=2cm,left=3cm,right=3cm,marginparwidth=1.75cm]{geometry}

\usepackage{graphicx}
\usepackage[colorlinks=true, allcolors=blue]{hyperref}
\usepackage{comment}
\usepackage{amsmath,amsthm,amssymb,amsfonts,amscd, amsbsy,mathtools}
\usepackage{float}
\usepackage{mathrsfs}

\newtheorem{theorem}{Theorem}[section]
\newtheorem{proposition}{Proposition}[section]

\newtheorem{definition}{Definition}[section]

\newtheorem{example*}{Example}
\newtheorem{remark}{Remark}
\newtheorem{remark*}{Remark}
\newtheorem{lemma*}{Lemma}
\newtheorem{theorem*}{Theorem}
\newtheorem{proposition*}{Proposition}
\newtheorem{corollary*}{Corollary}
\newtheorem{definition*}{Definition}
\newtheorem*{acknowledgment*}{Acknowledgment}

\newcommand{\C}{\mathbb{C}}
\newcommand{\R}{\mathbb{R}}
\renewcommand{\S}{\mathbb{S}}
\renewcommand{\H}{\mathbb{H}}
\newcommand{\g}[1]{\ensuremath{\mathfrak{#1}}}

\DeclareMathOperator{\Isom}{Isom}

\DeclareMathOperator{\Ad}{Ad}
\DeclareMathOperator{\ad}{ad}

\DeclareMathOperator{\SL}{SL}
\DeclareMathOperator{\SO}{SO}
\DeclareMathOperator{\Sol}{Sol}
\DeclareMathOperator{\Nil}{Nil}
\DeclareMathOperator{\sech}{sech}
\DeclareMathOperator{\Span}{span}
\DeclareMathOperator{\Aut}{Aut}

\title[Homogeneous Hypersurfaces in 4-dim. Thurston Geometries with 4-dim. Isometry Group]{Homogeneous Hypersurfaces in 4-dimensional Thurston Geometries with 4-dimensional Isometry Group}
\author{Ferreira, T.A.}

\address{Tarcios Andrey Ferreira - Instituto de Ciências Matemáticas e de Computação, Universidade de São Paulo, 13566-590, São Carlos-SP, Brazil }
\email{t.a.ferreira@icmc.usp.br}

\begin{document}



\begin{abstract}
We classify, up to conjugacy, the 3-dimensional subalgebras of the Lie algebras associated with the 4-dimensional Thurston geometries whose isometry groups are 4-dimensional. Since homogeneous hypersurfaces arise as orbits of subgroups of the isometry group acting transitively on the ambient space, we determine all such subgroups and describe their corresponding orbits. Consequently, for every left-invariant metric on these geometries, we obtain a classification of homogeneous hypersurfaces, up to ambient isometries, and investigate the geometry of the corresponding orbit foliations.
\end{abstract} 

\maketitle

\section{Introduction}

The theory of hypersurfaces constitutes a central branch of classical Differential Geometry, devoted to the study of manifolds $M^n$ immersed or embedded into an ambient manifold $N^{n+1}$. The intrinsic geometry of $M^n$ is determined by the metric induced from the ambient space, whereas its extrinsic geometry is encoded by the second fundamental form, which measures the way the hypersurface bends inside $N^{n+1}$. The investigation of hypersurfaces satisfying distinguished geometric conditions remains an active and influential area of research in Differential Geometry.

Among the most important classes of hypersurfaces are the {\em homogeneous hypersurfaces}, namely those for which the isometry group of the ambient space acts transitively on the hypersurface. The study of such objects has developed into a major research topic over the last century, and the classification of homogeneous hypersurfaces in space forms dates back to the early twentieth century through contributions by Cartan, Levi-Civita, Segre, Somigliana, and Kobayashi. For example, the classical theorem of Kobayashi \cite{Ko1958} states that every compact homogeneous hypersurface in $\R^{n+1}$ is necessarily isometric to the sphere. 

This line of research naturally led to the study of hypersurfaces whose sufficiently close parallel hypersurfaces possess constant mean curvature, the so-called {\em isoparametric hypersurfaces}. The associated codimension-one foliations are called {\em isoparametric foliations}. Cartan proved that a hypersurface in a space form is isoparametric if and only if it has constant principal curvatures. In the Euclidean and hyperbolic settings, this condition also characterizes homogeneous hypersurfaces. The spherical case, however, is substantially more intricate, and a complete classification was achieved only recently through the work of Chi \cite{Chi2020}.

Over the past decades, considerable attention has also been devoted to the study of minimal and constant mean curvature surfaces and, more generally, hypersurfaces in homogeneous manifolds. Roughly speaking, such ambient spaces represent the simplest and most symmetric geometries after the space forms. In dimension three, they are closely related to Thurston's eight model geometries \cite{scott83}, namely $\R^3$, $\S^3$, $\H^3$, $\S^2 \times \R$, $\H^2 \times \R$, $\widetilde{\SL}(2,\R)$, $\Nil^3$, $\Sol^3$. These geometries can be  classified according to the dimension of their isometry group, which may be equal to $3$, $4$, or $6$.  The case of dimension $6$ corresponds to the space forms. When the dimension of the isometry group is $4$, one obtains the $\mathbb{E}(\kappa,\tau)$ spaces, which arise as total spaces of Riemannian submersions with bundle curvature $\tau$ over a simply connected complete surface $M^2(\kappa)$ of constant curvature $\kappa$. In all these spaces, the equivalence between isoparametric surfaces and surfaces with constant principal curvatures still holds \cite{MiguelManzano}. When the isometry group has dimension $3$, the ambient spaces are precisely the unimodular and non-unimodular Lie groups endowed with left-invariant metrics, excluding the space forms and the $\mathbb{E}(\kappa,\tau)$ spaces. In particular, the classification of homogeneous surfaces in homogeneous $3$-manifolds was completed in \cite{MiguelTarciosTomas}, providing an important foundation for our investigation.

The main purpose of this work is to classify homogeneous hypersurfaces in the four-dimensional Thurston geometries whose isometry group has dimension four, classified by Filipkiewicz in \cite{Filipkiewicz} (see Table~\ref{tab:Thurston_4d_geo} for the complete list of such geometries). Although no analogue of Thurston's geometrization theorem is known in dimension four, the study of four-dimensional geometries and their submanifolds remains highly relevant from the Riemannian point of view, as it contributes to a deeper understanding of the geometric and algebraic structure of these spaces.
\begin{table}[h]
    \centering
    \begin{tabular}{|c|c|c|}
    \hline
    Thurston geometry ($X$) & Isotropy & dim(Isom($X$)) \\ 
    \hline
        $\R^4$, $\S^4$, $\H^4$ & SO(4) & 10  \\
         $\C \mathbf{P}^2$, $\C \mathbf{H}^2$ & U(2) & 9 \\ 
         $\S^3 \times \R$, $\H^3 \times \R$ & SO(3) &  7 \\ 
         $\S^2 \times \R^2$, $\S^2 \times \S^2$, $\H^2 \times \S^2$, $\H^2 \times \R^2$, $\H^2 \times \H^2$ & $\SO(2)\times \SO(2)$ & 6  \\
         $\Sol_0^4$, $\widetilde{\SL}(2, \R) \times \R$, $\Nil^3 \times \R$ & $\SO(2)$ & 5 \\ 
         $\textnormal{F}^4$ & $(\S^1)_{1,2}$ & 5 \\ 
         $\Nil^4$, $\Sol_1^4$ , $\Sol_{m,n}^4$ & $\{1\}$ &  4 \\ 
    \hline
    \end{tabular}
    \caption{Four-dimensional Thurston geometries ordered by their isotropy and dimension of isometry group. Here, in $\Sol^4_{m,n}$, $m$ and $n$ are positive integers and when $m = n$, $\Sol^4_{m,n} \simeq \Sol^3 \times \R$. Moreover, $(\S^1)_{a,b}$ denotes the image of $\S^1$ in $\textnormal{U}(2) \subset \SO(4)$ by $z \mapsto (z^a, z^b)$. } \label{tab:Thurston_4d_geo}
\end{table}
Several four-dimensional geometries have already been investigated, particularly the real and complex space forms and products of lower-dimensional geometries (see \cite{DiazRamos}). The cases of $\R^4$, $\H^4$, and $\S^4$ follow from the classical results of Segre and Cartan on isoparametric hypersurfaces (see \cite{JSC,CeRyan}). The homogeneous hypersurfaces in $\S^3 \times \R$ and $\H^3 \times \R$ were studied in \cite{ManfioSantosVeken}; the cases of $\C\mathbf{P}^2$ and $\C\mathbf{H}^2$ were investigated in \cite{DIAZRAMOSComplexhyper,MiguelComplexPro}; and the geometries $\S^2 \times \R^2$, $\H^2 \times \S^2$, $\H^2 \times \R^2$, $\S^2 \times \S^2$, and $\H^2 \times \H^2$ were considered in \cite{joaojoao,Urbano,GaoMaYao}. Many of these works classify isoparametric hypersurfaces and, consequently, homogeneous hypersurfaces as well. Also, in a recent work \cite{dhaeneGuoxin} the authors classify hypersurfaces with constant principal curvatures in $\Sol^4_0$ and, as consequence, they completely classify homogeneous hypersurfaces. Nevertheless, the approach developed in the present work is substantially different.

Four-manifolds are considerably less understood than their three-dimensional counterparts, which is reflected in the absence of a complete geometrization theorem in dimension four. Nevertheless, the study of four-dimensional Thurston geometries remains a natural and important direction of research. In this work, we focus specifically on the four-dimensional geometries whose full isometry group has dimension four, namely $\Nil^4$,  $\Sol_1^4$ and $\Sol_{m,n}^4$.

\begin{remark}
    Table \ref{tab:Thurston_4d_geo} is a consequence of \cite[Theorem~3.2]{Wall}. In fact, let $X$ be a maximal four-dimensional Thurston geometry and $G_X$ its isometry group with identity component $G^o_X$. Denote by $K^o_X$ the isotropy group of $G^o_X$. Hence, $G_X$ will have stabilizer $K_X \subset O_4$ with identity component $K^o_X$. In the space forms, it is well known that $K_X = O_4$. If $ X =  \mathbb{P}^2(\C)$ or $ X = \H^2(\C)$, then $K_X$ is equal to the normalizer $\tilde{U}_2$ of $U_2$.

    In general, if we have a Riemannian product $A \times B$ where $A$ and $B$ are irreducible and not isometric to each other, we have $ \Isom (A \times B) = \Isom A \times \Isom B$. This solves the cases $\S^3 \times \R$, $\H^3 \times \R$, $\S^2 \times \R^2$, $\S^2 \times \S^2$, $\H^2 \times \S^2$, $\H^2 \times \R^2$, $\H^2 \times \H^2$,  $\Nil^3 \times  \R$, $\widetilde{\SL}(2,\R) \times \R$ and $\Sol^3 \times \R$.

    If $X = F^4$, any automorphism $\alpha$ of $G_X^o = \R^2 \rtimes \SL(2,\R)$ must leave $\R ^2$ and $\SL(2,\R)$ invariant. It turns out that there are only two cases for $\alpha$, and $G_X = \R^2 \rtimes \SL^\pm(2,\R)$, where the $\pm $ denotes that matrices may have determinant equal to $\pm 1$. If $ X = \Sol_0^4$, $G^o_X$ is an extension of $\R^3$ by $\R \times \SO_2$ and since it acts faithfully, any automorphism $\alpha$ preserves $\R^3$ and is determined by its effect on $\R^3$. Hence, one deduces that $K_X \cong O_2 \times O_1$.
    
    In the remaining cases, one calculates directly in the Lie algebra $\g{g}$. Remembering that for a compact automorphism group that acts semi-simply, any invariant subspace will have an invariant complement, then it is possible to choose a basis, for the Lie algebra $\g{g}$, that fits this context in each case.  The proof is given by computing the possible automorphisms. Moreover, the last three cases all have $K_X^o$ trivial and in $ \Sol_{m,n}^4$, $\Sol^3 \times \R$ is a sub-case where there is an extra automorphism.
\end{remark}

This work is organized as follows. In Section \ref{Section_general_algebras}, we introduce the mathematical framework of the paper, describe all left-invariant metrics on the four-dimensional Thurston geometries whose isometry groups are four-dimensional, and classify, up to conjugacy, the three-dimensional subalgebras of their Lie algebras. As a consequence, in Section \ref{Section_general_hypersur}, we describe, up to ambient isometries, the geometry of homogeneous hypersurfaces and their orbit foliations by analyzing their shape operators.

\section{Codimension one subalgebras of 4-dimensional Lie algebras}\label{Section_general_algebras}

We begin by presenting a classification of left invariant metrics on the 4-dimensional, simply connected, unimodular Lie groups. Indeed, in \cite{Thuong}, they present such a classification, up to automorphism, and this classification is equivalent to a classification of the inner products on 4-dimensional unimodular Lie algebras. Hence, let $\tilde{\g{M}}$ denote the set of inner products on a Lie algebra $\g{g}$. There is a natural left action of $\Aut(\g{g})$ on $\tilde{\mathfrak{M}}$, considered in \cite{Kodama}. So, by computing $\Aut(\g{g})$ for each 4-dimensional unimodular Lie algebra, it is possible to describe the orbit space
$$ \g{M} = \tilde{\g{M}}/\Aut(\g{g}). $$
If two inner products lie in the same $\Aut(\g{g})$ orbit, then they induce isometric left invariant metrics on the corresponding simply connected Lie group $G$. The converse holds when $\g{g}$ only has real roots (see \cite{Alekseevski}). Therefore, when $\g{g}$ has only real roots, $ \g{M}$ is simply the space of left invariant metrics up to isometry.  See Table \ref{tab:Metric_4d_geo} for dim($\g{M}$) in each case. It is also important to remark that, in general, the classification up to automorphism is finer than that up to isometry.

\begin{table}[H]
    \centering
    \begin{tabular}{|l|c|l|c|l|c|}
    \hline
    $G$ & dim($\g{M}$) & $G$ & dim($\g{M}$) & $G$ & dim($\g{M}$) \\ 
    \hline
        $\R^4$ & 0 & $\Sol^3 \times \R$ & 4 & $\widetilde{\Nil_3 \rtimes \S^1}$ & 5 \\
        $\Nil_3 \times \R$ & 1 & $\Sol'^4_0$ & 4 & $\textnormal{SU}(2,\R) \times \R$ & 6 \\ 
        $\Sol_0^4$ & 2 & $\Sol_\mu^4$ & 4 & $\widetilde{\SL(2, \R)} \times \R$ & 6 \\
        $\Nil^4$ & 3 & $\widetilde{\Isom_0(\R^2)} \times \R$ & 4 & & \\ 
        $\Sol^4_{m,n}$ & 4 &  $\Sol^4_1$ & 5 &  &  \\
    \hline
    \end{tabular}
    \caption{Dimension of $\g{M}$ (see \cite{Thuong})}
    \label{tab:Metric_4d_geo}
\end{table}

\begin{remark}
    There are different classifications of the 4-dimensional Lie algebras in the literature. See \cite{Biggs} for a classification (simply connected or not) of the connected 4-dimensional Lie groups. Also, the 4-dimensional solvable Lie algebras are classified  in Theorem 1.5 of \cite{Andrada}. 
\end{remark}

Now, we focus on the 4-dimensional unimodular case. More precisely, we will concentrate on the 4-dimensional, simply connected, unimodular Lie groups with isometry group of dimension $4$. The following Proposition follows from \cite{Thuong}, Theorems 3.2, 4.3 and 5.1.  
\begin{proposition}\label{Metric_4dim_metrics}
    Let $G$ be a $4$-dimensional, simply connected, unimodular Lie group whose isometry group is 4-dimensional. Then, in each of the cases below, one can construct an orthonormal basis $\{\tilde{E}_i\}_{i=1}^4$ such that every left-invariant metric on $G$ is isometric to one of the following.
    \begin{enumerate}
        \item If $G = \Nil^4$, then for $ d_1, d_{3} > 0, \ d_{2} \geq 0$, we have 
        \begin{equation}\label{theo_metrics_nil4}
            E_1 = d_1 e_1, \ E_2 = d_2 e_1 + d_3 e_2, \ E_3 = e_3, \ E_4 = e_4.
        \end{equation}
        \item If $G = \Sol^4_1$, then for $d_{1}, d_{5} > 0 , \ d_{2}, d_{4} \geq 0, \ d_{3} \in \R$, we have 
        \begin{equation}\label{theo_metrics_sol14}
            E_1 = d_{1} e_1, \ E_2 = d_{2} e_1 + e_2 , \ E_3 =  d_{3} e_1 + d_{4} e_2 + e_3 , \ E_4 = d_{5} e_4.
        \end{equation}
        \item If $G = \Sol^4_{m,n}$, then for $d_{4} > 0 , \ d_{1}, d_{3} \geq 0, \ d_{2} \in \R$, we have  
        \begin{equation}\label{theo_metrics_solmn}
            E_1 = e_1, \ E_2 = d_{1} e_1 + e_2 , \ E_3 =  d_{2} e_1 + d_{3} e_2 + e_3 , \ E_4 = d_{4} e_4.
        \end{equation}
    \end{enumerate}
\end{proposition}

In the next result, we classify codimension-one subalgebras of $\g{g}$ up to conjugacy and isometric automorphisms. We recall that two subalgebras $\g{h}_1$ and $\g{h}_2$ are conjugate if there exists $g \in G$ such that $\g{h}_1 = \Ad(g)\g{h}_2$, where $\Ad$ denotes the adjoint representation of $G$. Our explicit description in terms of an orthonormal basis of $\g{g}$ will allow us to compute the geometry of the orbits and determine whether the actions of two subgroups are orbit equivalent. Thus, let $\g{g} = \Span\{ E_1, E_2, E_3, E_4 \}$ be the Lie algebra of $G$ and $\g{h}$ a subalgebra given by 
\begin{equation}\label{eq_def_h_ABC}
    \g{h} = \Span \{ A , B , C \}, \quad \textnormal{ where } \quad A = \sum_{j=1}^4 a_{j} E_j, \quad B = \sum_{j=1}^4 b_{j} E_j, \quad C = \sum_{j=1}^4 c_{j} E_j. 
\end{equation}

\begin{theorem}\label{theo_algebras}
    Let $G$ be a $4$-dimensional simply connected unimodular metric Lie group with a $4$-dimensional isometry group, and let $\g{g}$ be its Lie algebra. In terms of the metric coefficients $d_i \in \R$, $i=1,2,3,4,5$. given by Proposition~\ref{Metric_4dim_metrics}, a $3$-dimensional subspace $\g{h} \subset \g{g}$ is a subalgebra if and only if it is conjugate to one of the following:
    \begin{enumerate}
        \item If $\g{g} = \g{nil}_4$, then $\g{h}$ belongs to the one-parameter family of subalgebras
        $$ \g{h}_\theta = \Span\bigl\{ E_1, E_2, \cos(\theta) E_3 + \sin(\theta) E_4 \bigr\}, \quad \theta \in (0, \pi].$$
        
        \item If $\g{g} = \g{sol}_1^4$, then
        \begin{enumerate}
            \item $\g{h}_0 = \Span\{ E_1, E_2, E_3 \}$;
            \item $\g{h}_1 = \Span\{ E_1, E_2, E_4 \}$;
            \item $\g{h}_2 = \Span\left\{ E_1, \dfrac{1}{\sqrt{1 + d_4^2}}(d_4 E_2 - E_3), E_4 \right\}$.
        \end{enumerate}
        
        \item If $\g{g} = \g{sol}_{m,n}^4$, then
        \begin{enumerate}
            \item $\g{h}_0 = \Span\{ E_1, E_2, E_3 \}$;
            \item $\g{h}_1 = \Span\{ E_1, E_2, E_4 \}$;
            \item $\g{h}_2 = \Span\left\{ E_1, \dfrac{1}{\sqrt{1 + d_3^2}}(d_3 E_2 - E_3), E_4 \right\}$;
            \item $\g{h}_3 = \Span\left\{ \cfrac{1}{\sqrt{1 + d_1^2}}(d_1 E_1 - E_2), \dfrac{1}{\sqrt{1 + d_2^2}}(d_2 E_1 - E_3), E_4 \right\}$.
        \end{enumerate}
    \end{enumerate}
\end{theorem}

To prove Theorem~\ref{theo_algebras}, we first construct an orthonormal basis of $\g{g}$ adapted to the metrics in Proposition~\ref{Metric_4dim_metrics} and express the Lie bracket of $\g{g}$ in this basis. Next, we determine the possible $3$-dimensional subalgebras $\g{h} \subset \g{g}$ using \eqref{eq_def_h_ABC} and row reduction, simplifying the resulting cases by normalizing the generating vectors when needed. Finally, where necessary, we apply the identity
\begin{equation}\label{eq_conjug}
    e^{\operatorname{ad}(X)}Y := \sum_{n=0}^\infty \frac{\operatorname{ad}(X)^n}{n!}(Y) = Y + [X,Y] + \frac{1}{2}[X,[X,Y]] + \cdots
\end{equation}
to classify the subalgebras up to conjugacy.

Following \cite{Thuong}, the nilpotent Lie algebra $\g{nil}_4$ has a basis $(e_i)$ such that 
    \begin{equation} \label{eq_lie_braket_nil4_basic}
    \begin{aligned}
         \left[ e_1,  e_2 \right] & = 0, &   [e_1 , e_3] & =  0, &  \left[e_1 , e_4\right] & = 0,  \\
         \left[ e_2 , e_3\right] & = 0, & [e_2 , e_4] & = e_1 , & [e_3 , e_4] & = e_2.   
    \end{aligned}
    \end{equation}
For the solvable Lie algebra $\g{sol}_1^4$, there is a basis $(e_i)$ such that 
\begin{equation}\label{eq_braket_sol14}
    \begin{aligned}
    \left[ e_1,  e_2 \right] & =  0 , &  [e_1 , e_3] & = 0, &  [e_1 , e_4] & = 0, \\
    \left[e_2 , e_3\right] & = e_1, & [e_2 , e_4] & = -e_2 , & [e_3 , e_4 ] & = e_3.
    \end{aligned}
\end{equation}

The next definition is guaranteed by \cite{Filipkiewicz}, \cite{Lee} and \cite{Wall}. 
\begin{definition}
    For $\lambda > 1$, the $4$-dimensional Lie group $\operatorname{Sol}_\lambda^4$ is defined as the semidirect product $\R^3 \rtimes_\varphi \R$, where
    $$\varphi(s) =
    \begin{bmatrix}
        e^{\lambda s} & 0 & 0 \\
        0 & e^{s} & 0 \\
        0 & 0 & e^{-(1 + \lambda) s}
    \end{bmatrix}.$$
    When $\lambda = \ln \alpha_1/\ln \alpha_2$, where for some $m, n \in \mathbb{Z}$, $\alpha_1 > \alpha_2 > \alpha_3$ are three distinct positive real roots of $x^3 - m x^2 + n x - 1 = 0$, the group $\operatorname{Sol}_\lambda^4$ is denoted by $\operatorname{Sol}_{m,n}^4$. There are only countably many such values of $\lambda$.
\end{definition}
The solvable Lie algebra $\g{g}$ of $G = \Sol_\lambda^4 = \Sol_{m,n}$ and $(e_i) $ be a basis for $\g{g}$, we have 
\begin{equation}\label{eq_braket_solmn_basic}
    \begin{aligned}
        \left[e_1, e_2\right] & = 0, & [e_1, e_3] & = 0, & [e_1,e_4] & = - \lambda e_1, \\   
        \left[e_2, e_3\right] & = 0, & \left[e_2,e_4\right] & = - e_2, &  [e_3,e_4] & = (\lambda + 1) e_3.
    \end{aligned}
\end{equation}

\begin{proof}[Proof of Theorem \ref{theo_algebras}] 
    Let $\g{g}$ be one of the Lie algebras $\g{nil}_4$, $\g{sol}_1^4$ or $\g{sol}_{m,n}^4$. We classify the three-dimensional subalgebras $\g{h}$ of $\g{g}$ in separate cases. In all the cases below consider $A, B$ and $C$ defined as in  \eqref{eq_def_h_ABC}.

    If $\g{g} = \g{nil}_4 $ we rewrite \eqref{eq_lie_braket_nil4_basic} using \eqref{theo_metrics_nil4} as 
    $$ \begin{aligned}
        \left[E_2 , E_4 \right] & =  p E_1  , \quad \left[E_3 , E_4 \right] = q E_2  + r E_1, \\   
        \left[ E_1,  E_2 \right] & = \left[E_1 , E_3 \right] =  \left[E_1 , E_4 \right] = \left[E_2,E_3 \right] = 0,
    \end{aligned} \quad \textnormal{ where } p = \frac{d_3}{d_1}, \ q = \frac{1}{d_3}, \ r = - \frac{d_2}{d_1 d_3}$$
    Hence, we have the following cases 
    \begin{enumerate}
        \item If $a_4 = b_4 = c_4 = 0 $ we have $\g{h}_0 = \Span\{ E_1, E_2, E_3\}$, that is trivially a subalgebra since $ [A,B] =  [A,C] =  [B,C] = 0 $. Suppose now that $a_4 \ne 0$ and $b_3 = b_4 = c_3 = c_4 = 0$. Thus, $\g{h}_1 = \Span\{E_1, E_2, A\}$, where $A = a_3 E_3 + a_4 E_4$. Since $[E_1, E_2] = [E_1, A] = 0$ and $[E_2, A] = a_4 p E_1$, it follows that $\g{h}_1$ is indeed a subalgebra. Moreover, because every subalgebra admits an orthonormal basis, we may normalize $A$ to be a unit vector. Since $a_4 \ne 0$, this yields $a_4 = \pm \sqrt{1 - a_3^2}$ with $a_3 \in (-1, 1)$. Setting $a = a_3$, this case reduces to the family of subalgebras $ \g{h}_1^\pm(a) = \Span\bigl\{ E_1, E_2, a E_3 \pm \sqrt{1 - a^2}\, E_4 \bigr\}$, where $a \in (-1, 1)$.
        
        Suppose by contradiction, that $\g{h}_1^+(a)$ and $\g{h}_1^-(a)$ are conjugate for some real $a$ such that $ 0 < |a| < 1$. Observing that $a E_3 - \sqrt{1 - a^2}\, E_4 \notin \g{h}_1^+(a)$, there must exist elements $X \in \g{g}$ and $Y \in \g{h}_1^+(a)$, with $X = \sum_{i=1}^4 x_i E_i$ and $Y = \sum_{i=1}^4 y_i E_i$, such that $e^{\operatorname{ad}(X)} Y = a E_3 - \sqrt{1 - a^2}\, E_4$. A direct computation gives
        \begin{align*}
            [X, Y] &= (x_2 y_4 - x_4 y_2) p E_1 + (x_3 y_4 - x_4 y_3)(q E_2 + r E_1), \\
            [X, [X, Y]] &= -x_4 (x_3 y_4 - x_4 y_3) p q E_1,
        \end{align*}
        and all higher-order brackets vanish. In particular, $[X, Y]$ and $[X, [X, Y]]$ lie in $\Span \{E_1, E_2\}$. Since $Y \in \g{h}_1^+(a)$, we must have $(y_3, y_4) = t(a, \sqrt{1 - a^2})$ for some $t \in \R$, but on the other hand, by \eqref{eq_conjug}, $e^{\operatorname{ad}(X)} Y$ yields $y_3 = a$ and $y_4 = -\sqrt{1 - a^2}$, a contradiction. 
        
        Moreover, for $\theta \in (0,2\pi]$ we set $a = \cos\theta$ and combine $\g{h}_0$ with $\g{h}_1^+(a)$ and $\g{h}_1^-(a)$ (which corresponds to $\theta = 2 \pi$, $\theta \in (0, \pi)$ and $\theta \in (\pi, 2\pi)$, respectively), and excluding the cases by sign (since if $ V = \cos(\theta) E_3 + \sin(\theta) E_4  \in \g{h}$, then $- V \in \g{h}$), we conclude that every $3$-dimensional subalgebra belongs to the one-parameter family $\g{h}_\theta$, with $\theta \in (0,\pi]$.
        
        \item Now, consider the case where $a_4 \ne 0 $, $b_3 \ne 0$, $c_1 \ne 0$ and  $b_4 = c_4 = c_3 = c_2 = 0 $. We set $\g{h}_2 = \Span\{A,B,C\}$, where $A = a_2 E_2 + a_4 E_4$, $B = b_2 E_2 + b_3 E_3$ and $C = E_1$. Thus
        $$ \big[ A , B \big] = - a_4 ( b_3 r  + b_2 p ) E_1  - a_4 b_3 q E_2, \ \ \ \big[ A , C \big] = 0, \ \ \ \big[ B , C \big] = 0.$$
        Since $a_4 \ne 0$ and $b_3 \ne 0$ by hypothesis, this is never a subalgebra.

        \item Finally, suppose that $a_4 \ne 0 $, $b_3 \ne 0$, $c_2 \ne 0$ and  $b_4 = c_4 = c_3 =  0 $. We set $\g{h}_3 = \Span\{A,B,C\}$, where $A = a_1 E_1 + a_4 E_4$, $B = b_1 E_1 + b_3 E_3$ and $C = c_1 E_1 + c_2 E_2$. Thus 
        $$ [A,B] = - a_4 b_3 ( r  E_1 + q E_2 ), \quad  [A,C] =  - a_4 c_2 p E_1 , \quad  [B,C] =  0 . $$
        Since $p >0 $ , $a_4 \ne 0$ and $c_2 \ne 0 $, this is never a subalgebra.
    \end{enumerate}

    If $\g{g} = \g{sol}^4_1 $, we use equations \eqref{theo_metrics_sol14} and \eqref{eq_braket_sol14} to obtain  
    $$ \begin{aligned}
        \left[ E_1,  E_2 \right] & = \left[ E_1 , E_3 \right]  =  \left[ E_1 , E_4 \right]  = 0 , \\
        \left[ E_2 , E_3 \right] & = p E_1  , \quad  \left[ E_2 , E_4\right]  =  q E_1 + r E_2 , \quad \left[E_3 , E_4 \right]  =  s E_1 + t E_2 - r E_3 , 
    \end{aligned}$$
    where 
    $$ p = \frac{1}{d_1}, \ \ q = \frac{d_5 d_2}{d_1}, \ \ r = - d_5, \ \ s = (2d_4d_2 - d_3)\dfrac{d_5}{d_1} , \ \ t = -2 d_4 d_5 $$
    Thus, we have the following:
    \begin{enumerate}
        \item Suppose $a_4 = b_4 = c_4 = 0 $. We have $\g{h}_0 = \Span\{ E_1, E_2, E_3\}$, which is a subalgebra since $ [E_1,E_2] = [E_1,E_3] = 0$ and $[E_2,E_3] = p E_1$.
    
        \item Suppose $a_4 \ne 0  $ and $b_4 = c_4 = b_3 = c_3 = 0 $. We set $\tilde{\g{h}}_1 = \Span\{E_1,E_2,A\}$, where $A = a_3 E_3 + a_4 E_4  $. Thus $ [E_1,E_2] = [E_1,A] =  0$ and $ [E_2,A] = (a_3 p + a_4 q ) E_1 + a_4 r E_2 $, which implies that $\tilde{\g{h}}_1$ is a subalgebra.

        \item Suppose $a_4 \ne 0 $, $b_3 \ne 0$, $c_1 \ne 0$ and  $b_4 = c_4 = c_3 = c_2 = 0 $. This way, set $\tilde{\g{h}}_2 = \Span\{A,B,C\}$ where  $A = a_2 E_2 + a_4 E_4$, $B = b_2 E_2 + b_3 E_3$ and $C = E_1$. Thus
        $$ \begin{aligned}
            \left[ A , B \right] & = (a_2 b_3 p - a_4 b_2 q - a_4 b_3 s ) E_1 - (a_4 b_2 r  + a_4 b_3 t ) E_2 + a_4 b_3 r E_3     \\
            \left[ A , C \right] & =  \left[ B , C \right] = 0. 
        \end{aligned}$$
        Thus, $\tilde{\g{h}}_2$ is a subalgebra if there exists $ \alpha \in \R$ such that   
        $$ \alpha (b_2 E_2 + b_3 E_3)= - (a_4 b_2 r  + a_4 b_3 t ) E_2 + a_4 b_3 r E_3. $$
        Since $a_4 \ne 0$, $b_3 \ne 0$ and $r \ne 0 $ by hypothesis, we have $ b_2 = - b_3 d_4.$
    
        \item Suppose $a_4 \ne 0 $, $b_3 \ne 0$, $b_4 = c_4 = c_3 =  0 $ and set $\tilde{\g{h}}_3 = \Span\{A,B,C\}$, where $A = a_1 E_1 + a_4 E_4$, $B = b_1 E_1 + b_3 E_3$ and $C = c_1 E_1 + c_2 E_2$. Thus
        $$  \big[ A , B \big] = - a_4 b_3 (s E_1 + t E_2 - r E_3), \ \ \ \big[ A , C \big] = - a_4 c_2 (q E_1 + r E_2),  \ \ \big[ B , C \big] = - b_3 c_2 p E_1 .$$
        Since $p >0 $ , $b_3 \ne 0$ and $c_2 \ne 0$, then $\tilde{\g{h}}_3$ is never a subalgebra. 
    \end{enumerate}
    We now use equation \eqref{eq_conjug} to exclude subalgebras up to conjugation. 
    
    For the algebra  $ \tilde{\g{h}}_1 = \Span\{ E_1, E_2, a  E_3 + b E_4 \} $, if $b = 0$ then the subalgebra is $\g{h}_0 = \Span\{ E_1, E_2,  E_3 \} $. Thus, suppose that $b \ne 0$. Since $b \ne 0$ and $r \ne 0$, we choose $\psi = a/br$ to obtain $ e^{\ad(\psi E_3)} \tilde{\g{h}}_1 = \Span \{ E_1 , E_2 , E_4 \} $.
    Thus, up to conjugacy, this case simplifies to $\g{h}_1$.
    
    Now for $ \tilde{\g{h}}_2 = \Span \{ E_1,  d_4 E_2 - E_3 , a E_3 + b E_4 \} $. If $b = 0$ then this algebra is conjugate to $\g{h}_0 = \Span \{E_1, E_2, E_3\}$. Thus, suppose that $b \ne 0$. Since $b \ne 0$ and $r \ne 0$, we choose $\psi = a/br$ to obtain $ \hat{\g{h}}_2 = e^{\ad(\psi E_3)} \tilde{\g{h}}_2 = \Span \left\{E_1, d_4 E_2 - E_3 , (at/r) E_2 + b E_4 \right\}$. Now, using \eqref{eq_conjug} again and choosing $\phi = - at/br^2$ one gets $ e^{\ad(\phi E_2)}\hat{\g{h}}_2 =  \Span \left\{E_1, d_4 E_2 - E_3 , E_4 \right\}$.
    Thus this case simplifies to $\g{h}_2 $.

    Direct computations, as in the first case of $\g{g} = \g{nil}^4$, show that no two of the subalgebras $\g{h}_i$ are conjugate to each other.
   
    If $\g{g} = \g{sol}^4_{m,n} $, we use equations \eqref{theo_metrics_solmn} and \eqref{eq_braket_solmn_basic}, to obtain 
    \begin{equation}\label{eqcolcheteSol4mn}
        \begin{aligned}
            \left[ E_1,  E_2 \right] & =  \left[  E_1 , E_3 \right] =  \left[  E_2,E_3 \right] = 0 \\ 
            \left[ E_1 , E_4 \right] & = p E_1 , \quad  \left[ E_2 , E_4 \right] = q E_1 + r E_2 , \quad  \left[ E_3 , E_4 \right] =  s E_1 + t E_2 + u E_3 , 
        \end{aligned}
    \end{equation}
    where  
    $$\begin{aligned}
        p & = - \lambda d_4, &  q & = -d_4 d_1 (\lambda-1),  & r & = - d_4 , \\ 
        s & = d_4[d_1d_3(\lambda + 2 ) - d_2 (2 \lambda + 1) ] ,   &  t & = -d_4d_3(\lambda + 2) , & u & = d_4(\lambda + 1).
    \end{aligned}$$
    \begin{enumerate}
        \item Suppose $a_4 = b_4 = c_4 = 0 $. Hence $\g{h}_0 = \Span\{ E_1, E_2, E_3\}$, which is a subalgebra since \eqref{eqcolcheteSol4mn} yields $ [E_1,E_2] = [E_1,E_3] =  [E_2,E_3] = 0 $.        
            
        \item Suppose $a_4 \ne 0  $, $b_4 = c_4 = 0 $ and also that $b_3 = c_3 = 0 $. Thus $\tilde{\g{h}}_1= \Span\{ E_1 , E_2, A \} $, where $A = a_3 E_3 + a_4 E_4  $. Hence $\tilde{\g{h}}_1$ is a subalgebra since 
        $$ [E_1,E_2] =  0, \quad [E_1,A] =  p a_4 E_1, \quad  [E_2,A] = a_4 (q E_1 + r E_2) . $$ 
        
        \item Suppose $a_4 \ne 0 $, $b_3 \ne 0$, $c_1 \ne 0 $ and  $b_4 = c_4 = c_3 = c_2 = 0 $. We set $\tilde{\g{h}}_2 = \Span\{A,B,C\}$, where $ A =  a_2 E_2 + a_4 E_4$, $B = b_2 E_2 + b_3 E_3 $ and $ C = E_1 $. Thus
        $$ [ A , B \big] = a_4[ ( b_2 q +  b_3 s) E_1 + ( b_2 r + b_3 t )E_2 +  b_3 u E_3], \ \  [ A , C ] =  a_4 c_1 p E_1 ,  \ \ [ B , C  ] = 0 . $$
        Hence, $\tilde{\g{h}}_2$ is a subalgebra if and only if there exists $\alpha $  such that 
        $$ \alpha (b_2 E_2 + b_3 E_3)  =  ( b_2 r + b_3 t )E_2 +  b_3 u E_3 $$
        Since $b_3 \ne 0 $ by hypothesis, we have $\alpha = u$ and  $  b_2 ( u - r ) =  b_3 t   $. Also, as $\lambda > 1 $, then 
        $$  b_2  =  - b_3 \frac{ d_3 ( \lambda + 2 )}{( \lambda + 2 )} = - b_3 d_3  . $$ 
        Thus we have $ \tilde{\g{h}}_2 = \Span\{E_1, d_3 E_2 - E_3 , a_2 E_2 + a_4 E_4   \} $.
        
       \item Suppose that $a_4 \ne 0 $, $b_3 \ne 0$ and  $b_4 = c_4 = c_3 =  0 $ and $c_2 \ne 0 $. We now set $\tilde{\g{h}}_3 = \Span\{A,B,C\}$, where $ A = a_1 E_1 + a_4 E_4 $, $B = b_1 E_1 + b_3 E_3$ and $C = c_1 E_1 + c_2 E_2  $. Thus   
        $$ \begin{aligned}
            \left[ A , B \right] & = a_4 [ (b_1 p + b_3 s  ) E_1 +  b_3 t  E_2 + b_3 u E_3 ] , \\
            \left[ A , C \right] & =  a_4 [ (c_1 p + c_2 q ) E_1 + c_2 r E_2  ] , \\  
            \left[ B , C \right] & = 0.
        \end{aligned}$$
        Hence, $\tilde{\g{h}}_3$ is a subalgebra if there exists $\alpha $, $\beta$ and $\delta$ such that 
        $$\begin{aligned}
            \alpha ( b_1 E_1 + b_3 E_3) + \beta (c_1 E_1 + c_2 E_2) & = (b_1 p + b_3 s  ) E_1 +  b_3 t  E_2 + b_3 u E_3, \\
             \delta (c_1 E_1 + c_2 E_2) & = (c_1 p + c_2 q ) E_1 + c_2 r E_2.
        \end{aligned}$$
        Since $c_2 \ne 0$ and $\lambda > 1$, from the second equation we obtain $\delta = r $, which implies $ c_1 = - c_2 d_1 $. Moreover, since $b_3 \ne 0$ and $c_2 \ne 0$, the first equation yields $\alpha = u$ and $ \beta = b_3t/c_2 $. Thus we have $ b_1  =  -d_2 b_3  $. Hence   
        $\tilde{\g{h}}_3 = \Span\{ a_1 E_1 + a_4 E_4 ,d_2 E_1 - E_3 , d_1 E_1 - E_2 \}.$
    \end{enumerate}
    We now use equation \eqref{eq_conjug} to exclude subalgebras up to conjugation.
        
    Consider initially $\tilde{\g{h}}_1 = \Span\{ E_1, E_2, a_3 E_3 + a_4 E_4 \}$. If $a_3 = 0$ there is nothing to do. Thus, suppose $a_3 \ne 0 $. As $a_4 \ne 0$ and $u = d_4 (\lambda + 1 ) \ne 0$ by hypothesis, choose $\phi = - a_3/a_4u$ to obtain $e^{\ad(\phi E_3)} \tilde{\g{h}}_1 = \left\{ E_1, E_2, E_4 \right\}$.
    Thus, up to conjugacy, this case simplifies to $\g{h}_1$.
        
    For $ \tilde{\g{h}}_2 = \Span\{ E_1 , d_3 E_2 - E_3 , a_2 E_2 + a_4 E_4  \}$. If $a_2 = 0$ there is nothing to do. Thus, suppose $a_2 \ne 0 $. Since $a_4 \ne 0$ and $r \ne 0$ by hypothesis, we choose $\phi = - a_2/a_4r$ to get  
    $$ e^{\ad(\phi E_2)} \tilde{\g{h}}_2 = \Span \left\{ E_1 , C E_2 - E_3 , - a_2 d_1(\lambda - 1) E_1 + a_4 E_4  \right\}, $$
    Hence, this case simplifies to $\g{h}_2$.
        
    Finally, we consider $ \tilde{\g{h}}_3 = \Span \{  d_1 E_1 - E_2 , d_2 E_1 - E_3 , a_1 E_1 + a_4 E_4 \}$. One can see that, if $a_1 = 0$, there is nothing to do.  Suppose that $a_1 \ne 0 $. Since $a_4 \ne 0 $ and  $p = - \lambda d_4  \ne 0$, we set  $\phi = - a_1/a_4 p $ to obtain $ e^{\ad(\phi E_1)} \tilde{\g{g}}_3 = \Span\{  d_1 E_1 - E_2 , d_2 E_1 - E_3 , E_4 \}$, simplifying this case to $\g{h_3}$.

    Again, as in the first case of $\g{g} = \g{nil}^4$, one verifies by direct computation that none of the subalgebras $\g{h}_i$ are conjugate to one another.
    
\end{proof}

\section{Homogeneous hypersurfaces of 4-dimensional Thurston geometries}\label{Section_general_hypersur}
We begin this section with the following result.

\begin{proposition}\label{prop_relatdim4}
	Let $G$ be a four-dimensional simply connected metric Lie group whose full isometry group $\tilde{G} = \Isom(G,g)$ is also four-dimensional. Then isometric actions of connected subgroups of $\tilde{G}$ on $G$, up to orbit equivalence, are in one-to-one correspondence with Lie subalgebras of $\g{g}$ up to conjugacy and isometric automorphisms.
\end{proposition}

\begin{proof}
    Let $G$ be a four-dimensional simply connected Lie group endowed with a left-invariant metric $g$. Since the action of $G$ on itself by left translations is free, proper, and isometric, any Lie subgroup $H \subset G$ acts on $G$ with orbits of dimension $\dim H$. 
    
    For any metric Lie group, the full isometry group $\tilde{G} = \Isom(G,g)$ decomposes as $\tilde{G} = G \cdot K_e$, where $G$ is identified with the subgroup of left translations and $K_e$ denotes the isotropy subgroup at the identity element $e \in G$. Since $\dim \tilde{G} = 4$, the isotropy representation identifies $K_e$ with a finite subgroup of $\textnormal{O}(4)$. 
    
    Let $\Aut(G)$ denote the automorphism group of $G$. By \cite[Corollary 2.8]{HaLee}, the isotropy subgroup consists of the isometric automorphisms of $G$, that is, $K_e = \Aut(G) \cap \tilde{G}$. Consequently,
    $$ \tilde{G} = G \rtimes \left( \Aut(G) \cap \tilde{G} \right).$$
    In particular, the identity component of $\tilde{G}$ coincides with $G$, so every effective isometric action of a connected Lie group on $G$ is orbit equivalent to the action of a connected Lie subgroup $H \subset G$ by left translations. Moreover, denoting by $C_x$ the conjugation by $x \in G$, two connected subgroups $H, \widetilde{H} \subset G$ induce orbit equivalent actions if and only if there exist $x \in G$ and $\alpha \in \Aut(G) \cap \tilde{G}$ such that $H = C_x \circ \alpha(\widetilde{H})$. Passing to the Lie algebra $\g{g}$ via the standard correspondence between connected Lie subgroups of $G$ and Lie subalgebras of $\g{g}$ yields the desired bijection.
\end{proof}

If $G$ is a simply connected $n$-dimensional Lie group whose maximal isometry group (among all left-invariant metrics on $G$) is also $n$-dimensional, then $\dim \Isom(G,g) = n$ for every left-invariant metric $g$ on $G$. Indeed, for any left-invariant metric $g$, left translations act isometrically on $(G,g)$, so $G \subseteq \Isom(G,g)$ and hence $\dim \Isom(G,g) \geq n$. On the other hand, maximality implies $\dim \Isom(G,g) \leq n$, forcing equality.

Since four-dimensional Thurston geometries with a four-dimensional maximal isometry group are simply connected Lie groups, and every left-invariant metric on such a group $G$ is isometric to one of the metrics in Proposition~\ref{Metric_4dim_metrics}, we can apply Proposition~\ref{prop_relatdim4} to classify homogeneous hypersurfaces in these geometries.
     
Now, let $\gamma$ be a unit-speed normal geodesic to $H$ with $\gamma(0) = e$ and $\gamma'(0) = v$ and assume that $\gamma$ is defined in all $\R$. Let $\xi\in \g{g}$ be the left-invariant field with $\xi_e = v$, and write $H^t$ for the parallel displacement of $H$ in the direction of $\xi$ at distance $t$, that is, $H^t = \{\exp_h(t\xi_h):h\in H\}$ is the parallel (or equidistant) surface to $H$ at distance $t$, where $\exp$ is the Riemannian exponential map. Hence, $H^t=H\cdot \gamma(t)$, and the left-invariant vector field $\xi^t$ with $\xi^t_{\gamma(t)}=\gamma'(t)$ is a unit normal field to $H^t$. 
     
To obtain the shape operator $S^t$ of $H^t$ with respect to $\xi^t$, we need to determine the tangent vector $\gamma'$ to the normal geodesic $\gamma$, and then use the Koszul formula to find the Levi-Civita connection $\overline{\nabla}$ of $G$ in terms of left-invariant fields to obtain $S^t=-\overline{\nabla}_{\cdot\,} \xi^t$ for each $t\in\R$. 
     
Now, let $\{E_1,E_2,E_3,E_4\}$ be the left-invariant orthonormal frame defined previously, and write $\gamma'(t) = x(t) E_1 + y(t) E_2 + z(t) E_3 + w(t) E_4 $. We must have  
\begin{equation}\label{eq_geode_4dim}
    \overline{\nabla}_{\gamma'} \gamma' = x' E_1 + x \overline{\nabla}_{\gamma} E_1 + y' E_2 +  y \overline{\nabla}_{\gamma} E_2 + z' E_3 + z \overline{\nabla}_{\gamma} E_3 + w' E_4 +  w \overline{\nabla}_{\gamma} E_4 = 0. 
\end{equation}
where $\overline{\nabla}_{\gamma} E_i = x \overline{\nabla}_{E_1} E_i + y \overline{\nabla}_{E_2}  E_i + z \overline{\nabla}_{E_3} E_i +  w \overline{\nabla}_{E_4} E_i $. 
Equation \eqref{eq_geode_4dim} leads to a system of differential equations that we solve to describe $\gamma'(t)$.  

\begin{remark}
    In the next theorem, minimality follows from computing the shape operator along each orbit of an equidistant foliation of homogeneous hypersurfaces. However, it is worth noting that minimality also follows directly from the unimodularity of $G$. Indeed, let $\g{h} \subset \g{g}$ be the Lie algebra of a codimension-one Lie subgroup $H \subset G$, where $G$ is endowed with a left-invariant Riemannian metric. Choose a unit vector $N \in \g{h}^\perp$ and an orthonormal basis $\{e_1, \ldots, e_{n-1}\}$ of $\g{h}$, and extend them to left-invariant vector fields on $G$. Since left translations by elements of $H$ preserve both the metric and the tangent space of $H$, the field $N$ is a unit normal vector field along $H$. Moreover, the Koszul formula  yields
    $$2\langle \overline{\nabla}_{e_i} e_i, N \rangle = -\langle e_i, [e_i, N] \rangle + \langle e_i, [N, e_i] \rangle = 2\langle [N, e_i], e_i \rangle.$$
    Furthermore, since $[N, N] = 0$, the mean curvature $\mathcal{H}$ of $H$ is given by
    $$\mathcal{H} = \sum_{i=1}^{n-1} \langle \overline{\nabla}_{e_i} e_i, N \rangle = \sum_{i=1}^{n-1} \langle [N, e_i], e_i \rangle + \langle [N, N], N \rangle = \operatorname{tr}(\operatorname{ad}_N).$$
    The unimodularity of $G$ implies $\operatorname{tr}(\operatorname{ad}_X) = 0$ for all $X \in \g{g}$. Setting $X = N$ yields $\mathcal{H} = 0$. Thus every codimension-one Lie subgroup of $G$ is minimal.
    
    Now consider the orbits $H \cdot \gamma(t)$ of a connected codimension-one Lie subgroup $H \subset G$ along a geodesic  $\gamma(t)$. Writing $ H \cdot \gamma(t) = \gamma(t) \cdot \left(\gamma(t)^{-1} \cdot H \cdot \gamma(t) \right)$, one can see that each orbit is a left translate of the conjugate codimension-one subgroup $H^t := \gamma(t)^{-1} \cdot H \cdot  \gamma(t)$. Since left translations are isometries and each $H^t$ is minimal, every orbit $H \cdot \gamma(t)$ is minimal. Thus, every codimension-one homogeneous hypersurface in $G$ is minimal.
    
    Moreover, if $\g{h}$ is an ideal of $\g{g}$ and since $H$ is connected, then $\operatorname{Ad}_{\gamma(t)^{-1}}(\g{h}) = \g{h}$, which implies $H^t = H$ for all $t$. In this case, every orbit is a left translate of $H$, and its shape operator is independent of $t$.

\end{remark}

\begin{theorem}
    Let $G$ be a $4$-dimensional simply connected unimodular metric Lie group with a $4$-dimensional isometry group, and let $\g{g}$ be its Lie algebra. Up to ambient isometries, any homogeneous hypersurface in $G$ is minimal and is given by an orbit $H \cdot \gamma(t)$ from the list below, where $\gamma(t)$ is a geodesic intersecting $H$ at $e$ and $S^t$ is the shape operator of $H \cdot \gamma(t)$.
    
    If $G = \Nil_4$ with $d_1,d_3 > 0, \ d_2 \geq 0 $ and $ \g{h} = \Span\{ E_1, E_2, \cos(\theta) E_3 + \sin(\theta) E_4\}$, $\theta \in (0,\pi]$, then no orbit is totally geodesic and  the shape operator $S^t$ of $H\cdot \gamma(t)$ is given by 
    \begin{equation}\label{eq_shape_nil4}
        S^t = \frac{1}{2 d_1 d_3} \begin{bmatrix}
        0 & \omega & -d_2 \\
        \omega & 0 & d_1 \\ 
        - d_2 &  d_1 & 0 
        \end{bmatrix}, \ \ \ \omega = \cos(\theta) d_3^2.
    \end{equation}
    
    If $G = \Sol_1^4$ with $  d_1, d_5 > 0, \ d_2 , d_4 \geq 0, \ d_3 \in \R$, then no orbit is totally geodesic and 
    \begin{enumerate}
        \item If $ \g{h} = \Span\{ E_1, E_2, E_3 \}$, then the shape operator $S^t$ of $H\cdot \gamma(t)$ is given by 
        \begin{equation}\label{eq_shape_sol14_1}
            S^t = d_5 \begin{bmatrix}
            0 & \omega & \zeta \\
            \omega & -1 & - d_4   \\ 
            \zeta & - d_4  & 1 
        \end{bmatrix}, \ \  \omega = \frac{ d_2}{2d_1}, \ \ \zeta = \dfrac{2d_4d_2 - d_3}{2d_1}. 
        \end{equation}
        \item If $ \g{h} = \Span\{ E_1, E_2, E_4 \}$, then the shape operator $S^t$ of $H\cdot \gamma(t)$ is given by  
        \begin{equation}\label{eq_shape_sol14_2}
            S^t = d_5 \begin{bmatrix}
                0 & \omega(t) & \zeta \\
                \omega(t) & \tanh(-d_5t) & - d_4  \\ 
                \zeta & -  d_4 & - \tanh(-d_5t) 
            \end{bmatrix}, 
        \end{equation}
        where
        $$ \omega(t) =  \frac{1- d_2d_5\sinh(-d_5 t)}{2d_1d_5 \cosh(-d_5t)} , \ \ \zeta = \dfrac{2d_4d_2 - d_3}{2d_1}. $$
            
        \item If $ \g{h} = \Span \bigg\{ E_1, E_4, \dfrac{1}{\sqrt{d_4^2 + 1 }} (d_4 E_2 - E_3)  \bigg\} $, then the shape operator $S^t$ of $H\cdot \gamma(t)$ is given by 
        \begin{equation}\label{eq_shape_sol14_3}    
            S^t = d_5\begin{bmatrix}
                0 & \omega(t) & \zeta \\
                \omega(t) &  \tanh(-d_5 t) & -d_4 \\ 
                \zeta & -d_4 & - \tanh(-d_5 t)
            \end{bmatrix},
        \end{equation}
        where 
        \begin{equation}\label{eq_omega_zeta_sol14_3}
            \omega(t) = \frac{\sqrt{1 + d_4^2} + d_5(d_3- d_2d_4)\sinh(-d_5 t)}{2d_1d_5\cosh(-d_5 t)}, \ \ \ \zeta = \frac{d_2(1 + 2d_4^2) - d_3 d_4 }{2 d_1 \sqrt{1 + d_4^2 }}.
        \end{equation}
        \end{enumerate}
        
        If $G = \Sol_{m,n}^4$ with $d_4 > 0, \ d_1 , d_3 \geq 0, d_2 \in \R$ and $\lambda > 1 $ then 
        \begin{enumerate}
            \item If $ \g{h} = \Span\{ E_1, E_2, E_3 \}$, then the shape operator $S^t$ of $H\cdot \gamma(t)$ is given by 
            \begin{equation}\label{eq_shape_solmn_1}
                S^t =  - \frac{d_4}{2} \begin{bmatrix}
                    -2 \lambda & - d_1 (\lambda-1) & \omega\\
                    - d_1 (\lambda-1) & - 2 & - d_3(\lambda + 2)  \\ 
                    \omega & - d_3(\lambda + 2)  & 2(\lambda + 1) 
                \end{bmatrix},
            \end{equation}
            where $ \omega = d_1d_3(\lambda + 2 ) - d_2 (2 \lambda + 1) $. 
            None of the orbits is totally geodesic. 
            \item If $ \g{h} = \Span\{ E_1, E_2, E_4 \}$, then the shape operator $S^t$ of $H\cdot \gamma(t)$ is given by 
            \begin{equation}\label{eq_shape_solmn_2}
                S^t = - \frac{d_4}{2} \begin{bmatrix}
                    2\lambda \tanh( ut) & d_1 (\lambda-1) \tanh( ut)  & \omega \\
                    d_1 (\lambda-1) \tanh( ut)  & 2  \tanh( ut)  & - d_3(\lambda + 2) \\ 
                    \omega & -d_3(\lambda + 2) & - 2 (\lambda + 1) \tanh( ut) 
                \end{bmatrix},
            \end{equation}
            where $ \omega = [d_1d_3(\lambda + 2 ) - d_2 (2 \lambda + 1) ],   \ \   u = d_4 (\lambda + 1) $. 
            
            An orbit is totally geodesic if and only if $t = d_2 = d_3 = 0 $.
            \item If $ \g{h} = \Span \left\{ E_1 , \dfrac{1}{\sqrt{1 + d_3^2}}(d_3 E_2 - E_3) ,  E_4  \right\}$, then the shape operator $S^t$ of $H\cdot \gamma(t)$ is given by  
            \begin{equation}\label{eq_shape_solmn_3}
                S^t = \frac{d_4}{2} \begin{bmatrix}
                    2 \lambda \tanh(d_4 t ) & \omega(t) & \zeta \\
                    \omega(t) & - 2 (\lambda + 1 ) \tanh(d_4 t )  & d_3 (\lambda + 2 ) \\ 
                    \zeta & d_3 (\lambda + 2 ) & 2 \tanh(d_4 t )
                \end{bmatrix}, 
            \end{equation}
            where 
            \begin{equation}\label{eq_omega_zeta_solmn_3}  
                \omega(t) =  \frac{( 1 + 2 \lambda )(d_2  - d_3 d_1) \tanh(d_4 t)}{\sqrt{1 + d_3^2}}, \ \ \zeta = \frac{d_1( 1 - \lambda) + d_1 d_3^2 (2 + \lambda) - d_2 d_3 (1 + 2 \lambda)}{\sqrt{1 + d_3^2 }} .
            \end{equation}
            An orbit is totally geodesic if and only if $ t = d_1 = d_3 = 0 $. 
            
            \item If $ \g{h} = \Span \left\{  \dfrac{1}{\sqrt{1 + d_1^2}} (d_1 E_1 - E_2) , \dfrac{1}{\sqrt{1 + d_2^2}}(d_2 E_1 - E_3) , E_4 \right\}$, and let $P = - \dfrac{1}{\lambda d_4 }$. Then the shape operator $S^t$ of $H\cdot \gamma(t)$ is given by  
            \begin{equation}\label{eq_shape_solmn_4} 
                S^t =  d_4 \begin{bmatrix}
                    \tanh(Pt) & \omega(t) & \zeta (t) \\
                    \omega(t) &  - (\lambda + 1 ) \tanh(Pt) & \rho \\ 
                    \zeta(t)  & \rho &  \lambda \tanh(Pt) 
                \end{bmatrix},
            \end{equation}
            where 
            \begin{equation}\label{eq_omega_zeta_solmn_4}    
            \begin{aligned}
                 \omega(t) = & \frac{1}{2} \bigg[\frac{  (1 + d_1^2 + d_2^2)  (d_3 (2 + \lambda)( 1 + d_1^2) - d_1d_2 (1 + 2 \lambda))}{\sqrt{(1 + d_1^2) (1 + d_1^2 + d_2^2) (d_2^2 + (1 + d_1^2) \tanh^2(Pt))}} \\ 
                 & \quad - \frac{ (1 + d_1^2) (d_3 (1 + d_1^2) -d_1 d_2 )  (2 + \lambda) \sech^2(Pt)}{\sqrt{(1 + d_1^2) (1 + d_1^2 + d_2^2) (d_2^2 + (1 + d_1^2) \tanh^2(Pt))}}\bigg], \\
                 \zeta (t) & = \frac{d_1  (\lambda - 1) \sqrt{1 + d_1^2 + d_2^2} \tanh(Pt)}{2\sqrt{d_2^2 + (1 + d_1^2) \tanh^2(Pt)}}, \quad \rho = - \frac{d_2  (1 + 2 \lambda)}{2 \sqrt{1 + d_1^2 }}.
            \end{aligned}
            \end{equation}
            An orbit is totally geodesic if and only if $ t =  d_2 = 0 $.
        \end{enumerate}
\end{theorem}
\begin{proof}
    We now begin by determining the Riemannian connection of $\Nil^4$, $\Sol_1^4$ and $\Sol_{m,n}^4$ with respect to the basis given by Proposition \ref{Metric_4dim_metrics}. If $G = \Nil^4$, using Koszul Formula yields 
    \begin{equation}\label{eq_conec_Nil4}
        \begin{aligned}
            \overline{\nabla}_{E_1} E_1 & = 0, & \overline{\nabla}_{E_1} E_2 & =  -\frac{p}{2} E_4 , \\
            \overline{\nabla}_{E_2} E_1 & = -\frac{p}{2} E_4 ,& \overline{\nabla}_{E_2} E_2 & = 0 ,  \\
            \overline{\nabla}_{E_3} E_1 & = -\frac{r}{2} E_4 , & \overline{\nabla}_{E_3} E_2 & = -\frac{q}{2} E_4 ,   \\
            \overline{\nabla}_{E_4} E_1 & = \frac{1}{2} (p E_2 + r E_3) , & \overline{\nabla}_{E_4} E_2 & = \frac{1}{2} (-p E_1 + q E_3) ,  \\
            \overline{\nabla}_{E_1} E_3 & = -\frac{r}{2} E_4 , & \overline{\nabla}_{E_1} E_4 & = \frac{1}{2} (p E_2 + r E_3) , \\ 
            \overline{\nabla}_{E_2} E_3 & = -\frac{q}{2} E_4 , & \overline{\nabla}_{E_2} E_4 & = \frac{1}{2} (p E_1 + q E_3 ) ,  \\
            \overline{\nabla}_{E_3} E_3 & = 0 ,  & \overline{\nabla}_{E_3} E_4 & = \frac{1}{2} (r E_1 + q E_2) , \\
            \overline{\nabla}_{E_4} E_3 & = \frac{1}{2} (-r E_1 - q E_2) , & \overline{\nabla}_{E_4} E_4 & = 0 .
        \end{aligned}
    \end{equation} 
    If $G = \Sol_1^4$, we have   
    \begin{equation}\label{eq_conec_Sol14}
        \begin{aligned}
            \overline{\nabla}_{E_1} E_1 & = 0, &  \overline{\nabla}_{E_3} E_1 & = \frac{1}{2} (p E_2 - s E_4), \\
            \overline{\nabla}_{E_1} E_2 & =  -\frac{1}{2} (p E_3 + q E_4), & \overline{\nabla}_{E_3} E_2 & = -\frac{1}{2} (p E_1 + t E_4),   \\
            \overline{\nabla}_{E_1} E_3 & = \frac{1}{2} (p E_2 - s E_4), & \overline{\nabla}_{E_3} E_3 & = r E_4 , \\
            \overline{\nabla}_{E_1} E_4 & = \frac{1}{2} (q E_2 + s E_3),  & \overline{\nabla}_{E_3} E_4 & = \frac{1}{2} (s E_1 + t E_2 - 2 r E_3),  \\
            \overline{\nabla}_{E_2} E_1 & = -\frac{1}{2} (p E_3 + q E_4) , & \overline{\nabla}_{E_4} E_1 & = \frac{1}{2} (q E_2 + s E_3),  \\
            \overline{\nabla}_{E_2} E_2 & = -r E_4, & \overline{\nabla}_{E_4} E_2 & = \frac{1}{2} (-q E_1 + t E_3),  \\
            \overline{\nabla}_{E_2} E_3 & = \frac{1}{2} (p E_1 - t E_4), & \overline{\nabla}_{E_4} E_3 & = - \frac{1}{2} (s E_1 + t E_2) , \\
            \overline{\nabla}_{E_2} E_4 & = \frac{1}{2} (q E_1 + 2 r E_2 + t E_3 ), & \overline{\nabla}_{E_4} E_4 & = 0.  
        \end{aligned}
    \end{equation}
    For $\Sol_{m,n}^4$, we obtain   
    \begin{equation}\label{eq_conec_Solmn}
        \begin{aligned}
            \overline{\nabla}_{E_1} E_1 & = -p E_4 & \overline{\nabla}_{E_3} E_1 & = -\frac{s}{2}  E_4, \\
            \overline{\nabla}_{E_1} E_2 & =  -\frac{q}{2} E_4, & \overline{\nabla}_{E_3} E_2 & = -\frac{t}{2} E_4, \\
            \overline{\nabla}_{E_1} E_3 & = -\frac{s}{2} E_4, & \overline{\nabla}_{E_3} E_3 & = - u E_4, \\
            \overline{\nabla}_{E_1} E_4 & = \frac{1}{2} (2p E_1 + q E_2 + s E_3), & \overline{\nabla}_{E_3} E_4 & = \frac{1}{2} (s E_1 + t E_2 + 2 u E_3), \\ 
            \overline{\nabla}_{E_2} E_1 & = -\frac{q}{2} E_4, & \overline{\nabla}_{E_4} E_1 & = \frac{1}{2} (q E_2 + s E_3), \\
            \overline{\nabla}_{E_2} E_2 & = -r E_4, & \overline{\nabla}_{E_4} E_2 & = \frac{1}{2} (-q E_1 + t E_3), \\
            \overline{\nabla}_{E_2} E_3 & = - \frac{t}{2} E_4, & \overline{\nabla}_{E_4} E_3 & = - \frac{1}{2} (s E_1 + t E_2) , \\
            \overline{\nabla}_{E_2} E_4 & = \frac{1}{2} (q E_1 + 2 r E_2 + t E_3 ), & \overline{\nabla}_{E_4} E_4 & = 0 .
        \end{aligned}
    \end{equation}
    We now approach each case separately:

    If $G = \Nil^4$, then $\g{h} = \Span \{ E_1, E_2, V \}$, with $V = \cos(\theta) E_3 + \sin(\theta) E_4$. We have that $\gamma'(0) = \sin(\theta) E_3 - \cos(\theta) E_4 $ is the unitary normal field to this subspace. Using \eqref{eq_conec_Nil4} we have  
    $$\overline{\nabla}_{\gamma'(0)} \gamma'(0) = \frac{1}{2} (\cos(\theta)\sin(\theta) - \cos(\theta)\sin(\theta) )(r E_1 + q E_2) = 0.$$ 
    It follows that  $\gamma'(t) = \gamma'(0)$ for all $t \in\R$. As the tangent space to $H\cdot\gamma(t)$ at $\gamma(t)$ is given by $\g{h}_{\gamma(t)}$, we have    
    $$ - \overline{\nabla}_{E_1} \gamma' = \frac{r}{2} V  +  \cos(\theta) \frac{p}{2} E_2, \quad - \overline{\nabla}_{E_2} \gamma' =  \frac{q}{2} V  + \cos(\theta) \frac{p}{2} E_1, \quad - \overline{\nabla}_{V} \gamma' = \frac{1}{2}  (r E_1 + q E_2 ). $$
    Hence $S^t$ is given by \eqref{eq_shape_nil4}.

    If $G = \Sol^4_1$ let $\gamma'(t) = x(t) E_1 + y(t) E_2 + z(t) E_3 + w(t) E_4 $. Using \eqref{eq_conec_Sol14} we have 
    \begin{align*}
        \overline{\nabla}_{\gamma'} E_1 & =  \frac{1}{2} [(zp  + wq )E_2 + (ws - yp) E_3 - ( zs + yq ) E_4 ], \\
        \overline{\nabla}_{\gamma'} E_2 & =  \frac{1}{2} [ - (z p  + w q )E_1 + (wt - xp) E_3 - ( xq + 2 y r + zt )E_4 ],   \\
       \overline{\nabla}_{\gamma'} E_3 & = \frac{1}{2} [ (y p - w s) E_1 + (xp -wt) E_2 + (2zr- yt - xs) E_4 ],   \\
       \overline{\nabla}_{\gamma'} E_4 & = \frac{1}{2} [ (yq + z s)E_1 + (xq + 2yr + zt )E_2 + (xs + yt - 2rz )E_3 ]. 
    \end{align*}
    Thus $(x,y,z,w)$ must be the solution of the following system of ODEs
    \begin{equation}\label{systemSol14}
     \begin{cases}
        \begin{aligned}
        2 x' & = 0, \\
        2 y' + 2 x (zp  + wq ) + 2yr & = 0, \\
        2 z' + 2x(ws - yp)  + 2 ywt - 2wrz & = 0, \\ 
        2 w' - 2x( zs + yq ) - 2y ( y r + zt ) +  2z^2r & = 0. 
        \end{aligned}
    \end{cases}   
    \end{equation}
    Now, for each subalgebra in Theorem \ref{theo_algebras}, we compute $\gamma'$.
    
    If $\g{h} = \g{h}_0$, then $\gamma'(0) = E_4 $ is the unitary vector that is orthogonal to this subspace. Hence $\overline{\nabla}_{\gamma'(0) } \gamma'(0) = 0 $. Thus $\gamma'(t) = \gamma'(0)$ for all $t \in \R$. Hence
    \begin{align*}
        - \overline{\nabla}_{E_1} \gamma'(t) & =  - \frac{1}{2} (q E_2  + s E_3), \\  
        - \overline{\nabla}_{E_2} \gamma'(t) & =  - \frac{1}{2} ( q E_1 +2r E_2 + t E_3), \\  
        - \overline{\nabla}_{E_3} \gamma'(t) & = - \frac{1}{2} (s E_1 + t E_2 - 2r E_3 ). 
    \end{align*}
    Thus, $S^t$ is given by \eqref{eq_shape_sol14_1}.
     
    If $\g{h} = \g{h}_1$, we have initial conditions $ x(0) = y(0) = w(0) = 0$ and $z(0) = 1 $. Assuming $x \equiv 0$ and $y \equiv 0$, system \eqref{systemSol14} reduces to 
    $$ \begin{cases}
        2z' - 2rwz = 0, \\
        2w' + 2rz^2 = 0.
    \end{cases}$$
    With $w = z'/rz$, this is equivalent to solving $ z''z - (z')^2 + r^2 z^4 = 0$. Imposing the initial conditions yields
    \begin{equation}\label{eq_geo_sol14_1}
        \gamma'(t) = \frac{1}{\cosh(rt)} E_3 - \tanh(rt) E_4.
    \end{equation}
    Now, using \eqref{eq_geo_sol14_1} with $ z(t) = 1/\cosh(r t)$ and $\ w(t) = - \tanh(rt) $, the orthogonal subspace to $\gamma'(t)$ at $t$ is generated by the orthonormal frame $\{E_1, E_2, V \}$ with $V = w(t) E_3 - z(t) E_4$. Thus 
    \begin{align*}
        - \overline{\nabla}_{E_1} \gamma'(t) & = - \frac{1}{2} [(z(t)p + w(t)q ) E_2  + s V ], \\
        - \overline{\nabla}_{E_2} \gamma'(t) & = - \frac{1}{2} [(z(t)p + w(t)q ) E_1 + 2 w(t) r E_2 + t V ], \\
        - \overline{\nabla}_{V} \gamma'(t) & =   - \frac{1}{2} \left[  s E_1 + t E_2 - 2 r w(t) V \right].
    \end{align*}
    Hence, $S^t$ is given by \eqref{eq_shape_sol14_2}.
     
    If  $ \g{h}  = \g{h}_2 $, we have initial conditions $x(0) = w(0) = 0$, $  y(0) = 1/\sqrt{d_4^2 + 1 }$ and $ z(0) = d_4/\sqrt{d_4^2 + 1 }$. Since $t = 2 d_4 r $ and assuming $x \equiv 0 $ and $z(t) = d_4 y(t)$, system \eqref{systemSol14} reduces to 
    $$ \begin{cases}
        \begin{aligned}
            y' + r w y & = 0, \\
            w' - r (1 + d_4^2) y^2 & = 0. 
        \end{aligned}
    \end{cases}$$
        With $w = - y'/ry$, this is equivalent to solving $ - y'' y + y'^2 - r^2(1 + d_4^2 ) y^4 = 0$. Imposing the initial conditions we obtain 
        \begin{equation}\label{eq_geo_sol14_2}
            \gamma'(t) =  \dfrac{1}{\sqrt{d_4^2 + 1 }} \frac{1}{\cosh(r t)} E_2 + \dfrac{d_4}{\sqrt{d_4^2 + 1 }} \frac{1}{\cosh(r t)} E_3 +  \tanh(rt) E_4 
        \end{equation}
    Using \eqref{eq_geo_sol14_2} with 
    $ y(t) = 1/\sqrt{1 + d_4^2}\cosh(r t)$, $ z(t) =  d_4/\sqrt{1 + d_4^2}\cosh(r t)$ and $w(t) = \tanh(rt)$, the orthogonal space to $\gamma'(t)$ at $t$ is generated by the orthonormal frame $\{ E_1,V,W\}$, where  
    $$\begin{aligned}
        V & =  \cosh(rt)[z (t) E_2 - y(t) E_3], \\ 
        W & = \cosh(rt) [y(t) w(t)  E_2 + z(t) w(t)  E_3 - (y^2(t) + z^2(t)) E_4] .
    \end{aligned}  $$
    Thus  
    \begin{align*}
        - \overline{\nabla}_{E_1} \gamma'(t) & = d_5 \left( \omega(t) V + \zeta W \right), \\
        - \overline{\nabla}_{V} \gamma'(t) & =  d_5 \left( \omega(t) E_1 + \tanh(-d_5 t) V - d_4  W \right) , \\
        - \overline{\nabla}_{W} \gamma'(t) & =  d_5 \left( \zeta E_1 - d_4 V  - \tanh(-d_5 t) W \right),
    \end{align*}
    where $\omega(t)$ and $ \zeta $ are given by \eqref{eq_omega_zeta_sol14_3}. Hence, $S^t$ is given by \eqref{eq_shape_sol14_3}.

    If $G = \Sol_{m,n}^4$, let $\gamma' = x(t) E_1 + y (t) E_2 + z(t) E_3 + w(t) E_4 $. Using \eqref{eq_conec_Solmn} we have 
    \begin{align*}
        \overline{\nabla}_{\gamma'} E_1 & = \frac{1}{2} [ (- 2 x p - y q - z s) E_4  + w (q E_2 + s E_3 ) ], \\
        \overline{\nabla}_{\gamma'} E_2 & =    \frac{1}{2} [ (- x q - 2 y r - z t) E_4  + w ( -q E_1 + t E_3 ) ],   \\
       \overline{\nabla}_{\gamma'} E_3 & =   \frac{1}{2} [ (- x s - y t + 2 z u) E_4 - w ( s E_1 + t E_2 ) ], \\
       \overline{\nabla}_{\gamma'} E_4 & =   \frac{1}{2} [ (2 x p + y q +  z s ) E_1 + (x q + 2 y r +  z t) E_2 + (x s + y  t +  2 z u) E_3 ]. 
    \end{align*}
    Thus $\gamma(t)$ is a geodesic if it satisfies the system
    \begin{equation}\label{systemSolmn}
     \begin{cases}
        \begin{aligned}
        2 x' - w (q y + zs) + w(2 p x  + q y + s z ) & = 0, \\
        2 y' + w (xq - zt)  + w(q x + 2 r y + t z) & = 0, \\
        2 z' + w (xs + yt) + w(s x  + t y + 2 u z ) & = 0, \\ 
        w' - p x^2 - r y^2 + u z^2 - q x y - s x z - t y z & = 0. 
        \end{aligned}
    \end{cases}   
    \end{equation}
    We now solve this system for each subalgebra of Theorem \ref{theo_algebras} separately.

    If $\g{h} = \g{h}_0 $  we have that $\gamma'(0) = E_4 $ is the normal vector to this subspace. Since  $\overline{\nabla}_{E_4 } E_4  = 0 $, then $ \gamma'(t) = \gamma'(0)$, for all $t \in \R$. Thus 
    \begin{align*}
        - \overline{\nabla}_{E_1} \gamma'(t) & =  - \frac{1}{2} (2p E_1 + q E_2 + s E_3), \\  
        - \overline{\nabla}_{E_2} \gamma'(t) & =  - \frac{1}{2} (q E_1 + 2 r E_2 + t E_3 ), \\ 
        - \overline{\nabla}_{E_3} \gamma'(t) & = - \frac{1}{2} (s E_1 + t E_2 + 2 u E_3).
    \end{align*}
    Hence, $S^t$ is given by \eqref{eq_shape_solmn_1}.

    If $\g{h} = \g{h}_1$ we have initial conditions $ x(0) = 0$, $y(0) = 0$, $ z(0) = 1$,  $w(0) = 0 $. Assuming $x \equiv 0 $ and $y \equiv 0$, system \eqref{systemSolmn} becomes 
    $$ \begin{cases}
        \begin{aligned}
            z' + u w z  & = 0 , \\ 
            w' + u z^2  & = 0 .
        \end{aligned}
    \end{cases}$$    
    With $w = - z'/uz$, this is equivalent to solving $ z'' z - z'^2 - u^2 z^4 = 0$. Imposing the initial conditions, we have   
    \begin{equation}\label{eq_geo_solmn_1}
         \gamma'(t) =   \frac{1}{\cosh(ut)} E_3 - \tanh(ut) E_4. 
    \end{equation}
    Using \eqref{eq_geo_solmn_1}, the orthogonal space to $\gamma'(t)$ at $t$ is generated by the orthonormal frame $\{ E_1, E_2, V \}$, with $ V = w(t) E_3 - z(t) E_4 $. Hence
    \begin{align*}
        - \overline{\nabla}_{E_1} \gamma'(t) & = - \frac{1}{2} [2 p w E_1 + q w E_2 + s V], \\
        - \overline{\nabla}_{E_2} \gamma'(t) & = - \frac{1}{2} [ q w E_1 + 2 r w E_2 + t V ], \\ 
        - \overline{\nabla}_{V} \gamma'(t) & =  - \frac{1}{2} [ s E_1 + t E_2 + 2uw V ].
    \end{align*}
    Thus, $S^t$ is given by \eqref{eq_shape_solmn_2}.
    
    If  $\g{h} = \g{h}_2$ we have initial conditions $ x(0) = w(0) = 0$, $ y(0) = 1/\sqrt{d_3^2 + 1 }$, and $ z(0) =  d_3/\sqrt{d_3^2 + 1 }$.  Assuming $x(t) \equiv 0 $ and $z(t) = d_3 y(t)$, system \eqref{systemSolmn} reduces to  
    $$  \begin{cases}
        \begin{aligned}
            (1 + d_3) y' + y w (r + t  + d_3 u ) & = 0, \\ 
            w' - (d_3^2 u + t d_3 + r) y^2 & = 0 .
        \end{aligned}
    \end{cases} $$
    With $w = - y'(1 + d_3)/[y(r + t + d_3 u)]$, this is equivalent to solving $ y'' y - y'^2 + \phi \varphi y^4 = 0$, where $ \phi = - (r + t + d_3 u)/(1 + d_3)$ and $ \varphi = - (d_3^2 u + d_3 t + r)$.
    Imposing the initial conditions, we obtain 
    \begin{equation}\label{eq_geo_solmn_2}
        \gamma'(t) =  \dfrac{1}{\sqrt{d_3^2 + 1 }} \frac{1}{\cosh(d_4 t)} (E_2  + d_3 E_3) - \tanh(d_4t) E_4. 
    \end{equation}
    Using \eqref{eq_geo_solmn_2}, the orthogonal space to $\gamma'(t)$ at $t$ is generated by the orthonormal frame $\{ E_1, V , 
    W\}$ with 
    $$\begin{aligned}
        V & =  \cosh(d_4 t)[z (t) E_2 - y(t) E_3], \\
        W & = \cosh(d_4 t) [y w  E_2 + z w  E_3 + (y^2 + z^2) E_4].
    \end{aligned}$$ 
    Thus
    \begin{align*}
        - \overline{\nabla}_{E_1} \gamma'(t) & = \frac{d_4}{2}  \left( 2 \lambda \tanh(d_4 t )  E_1 + \omega V + \zeta W \right), \\
        - \overline{\nabla}_{V} \gamma'(t) & = \frac{d_4}{2}  \left( \omega E_1 - 2 (\lambda + 1 ) \tanh(d_4 t ) V + d_3 (\lambda + 2 ) W \right),\\
        - \overline{\nabla}_{W}\gamma'(t) & = \frac{d_4}{2}  \left( \zeta E_1 + d_3 (\lambda + 2 ) V + 2 \tanh(d_4 t ) W \right), 
    \end{align*}
    where $\omega(t)$ and $\zeta$ are given by \eqref{eq_omega_zeta_solmn_3}. Hence, $S^t$ is given by \eqref{eq_shape_solmn_3}.
    
    If  $\g{h} = \g{h}_3 $, we have initial conditions 
    $$ x(0) = \dfrac{1}{\sqrt{d_1^2 + d_2^2 + 1 }}, \ \ y(0) = \dfrac{d_1}{\sqrt{d_1^2 + d_2^2 + 1 }}, \ \ z(0) =  \dfrac{d_2}{\sqrt{d_1^2 + d_2^2 + 1 }}, \ \ w(0) = 0 . $$
    Assuming $z(t) = d_2 x(t)$ and $y(t) = d_1 x $, system \eqref{systemSolmn} becomes
    $$ \begin{cases}
        \begin{aligned}
            \phi x' + w x \varphi & = 0, \\ 
        w' - \eta x^2 & = 0. 
        \end{aligned}
        \end{cases}, \textnormal{ where } 
    \begin{cases}
        \begin{aligned}
            \phi & =  1 + d_1 + d_2 , \\ 
            \varphi & = p + q + s + (r + t )d_1 + u d_2 , \\ 
            \eta & = p  + r d_1^2  + q d_1 + t d_1 d_2 + s d_2 + u d_2^2.
        \end{aligned}
    \end{cases}$$
    With $w = - \phi x'/(\varphi x)$, this is equivalent to solving 
    $ x'' x - x'^2 + (\varphi \eta/\phi) x^4 = 0$. Imposing the initial conditions we get  
    \begin{equation}\label{eq_geo_solmn_3}
    \gamma'(t) =  \dfrac{1}{\sqrt{d_1^2 + d_2^2 + 1}} \dfrac{1}{\cosh(P t)}(E_1 + d_1 E_2 + d_2 E_3 ) + \tanh(Pt) E_4, \quad \textnormal{ where } P = - \frac{1}{\lambda d_4} .
    \end{equation}
    Using \eqref{eq_geo_solmn_3}, the orthogonal space to $\gamma'(t)$ at $t$ is generated by the orthonormal frame $\{V, W , Z\}$ with $ \mu = \sqrt{ x^2 + y^2}$, $ \nu = \sqrt{z^2 + w^2 }$, and 
    \begin{align*}
        V & = \dfrac{1}{\mu} [y (t) E_1 - x(t) E_2], \\ 
         W & =  \dfrac{1}{\nu} (w(t) E_3 -  z(t) E_4) , \\
         Z & = \dfrac{\nu}{\mu} (-  x E_1 - y E_2) +  \frac{\mu}{\nu} ( z E_3 + w  E_4).
    \end{align*}
    Hence
    $$\begin{aligned}
        - \overline{\nabla}_{V} \gamma'(t) & = d_4 \left( \tanh(Pt) V + \omega W + \zeta Z \right) , \\
        - \overline{\nabla}_{W} \gamma'(t) & = d_4 \left( \omega V - (\lambda + 1 ) \tanh(Pt) W + \rho Z \right) , \\
        - \overline{\nabla}_{Z} \gamma'(t) & = d_4 \left( \zeta V + \rho W + \lambda \tanh(Pt) Z \right) ,
    \end{aligned}$$
    where $\omega(t)$, $\zeta (t)$ and $\rho$ are given by \eqref{eq_omega_zeta_solmn_4}.  Thus, $S^t$ is given by \eqref{eq_shape_solmn_4}.    
\end{proof}
\section*{Declarations}
    \subsection*{Acknowledgment}
     The author would like to thank M. Domínguez-Vázquez for his valuable comments and insightful suggestions, which significantly improved the quality and presentation of this manuscript. 
     \subsection*{Funding}
     This study was financed, in part, by the São Paulo Research Foundation (FAPESP), Brasil. Process Number \# 2025/28626-8, by CNPq - Conselho Nacional de Desenvolvimento Científico e Tecnológico grant number 162148/2021-6 and CAPES Finance Code 001.

     The Article Processing Charge for the publication of this research was funded by the Coordenação de Aperfeiçoamento de Pessoal de Nível Superior - CAPES (ROR identifier: 00x0ma614). For the purpose of open access, the author have applied a Creative Commons CC BY license to any Author Accepted Manuscript version arising from this submission.
     
    \subsection*{Competing Interests}
    The author has no relevant financial or non-financial interests to disclose.
    \subsection*{Data Availability}
    Data sharing is not applicable to this article, as no datasets were generated or analyzed during the study.

\bibliographystyle{plain}
\bibliography{ref}

@article {Ko1958,
    AUTHOR = {Kobayashi, S.},
     TITLE = {Compact homogeneous hypersurfaces},
   JOURNAL = {Trans. Amer. Math. Soc.},
  FJOURNAL = {Transactions of the American Mathematical Society},
    VOLUME = {88},
      YEAR = {1958},
     PAGES = {137--143},
      ISSN = {0002-9947,1088-6850},
   MRCLASS = {53.00},
  MRNUMBER = {96284},
MRREVIEWER = {W.\ M.\ Boothby},
       DOI = {10.2307/1993242},
       URL = {https://doi.org/10.2307/1993242},
}

@article {Chi2020,
    AUTHOR = {Chi, Q.-S.},
     TITLE = {Isoparametric hypersurfaces with four principal curvatures,
              {IV}},
   JOURNAL = {J. Differential Geom.},
  FJOURNAL = {Journal of Differential Geometry},
    VOLUME = {115},
      YEAR = {2020},
    NUMBER = {2},
     PAGES = {225--301},
      ISSN = {0022-040X,1945-743X},
   MRCLASS = {53C42 (53C40)},
  MRNUMBER = {4100704},
MRREVIEWER = {Ruy\ Tojeiro},
       DOI = {10.4310/jdg/1589853626},
       URL = {https://doi.org/10.4310/jdg/1589853626},
}

@article{scott83,
    AUTHOR = {Scott, Peter},
     TITLE = {The geometries of {$3$}-manifolds},
   JOURNAL = {Bull. London Math. Soc.},
  FJOURNAL = {The Bulletin of the London Mathematical Society},
    VOLUME = {15},
      YEAR = {1983},
    NUMBER = {5},
     PAGES = {401--487},
      ISSN = {0024-6093},
   MRCLASS = {57N10 (22E10 53C20)},
  MRNUMBER = {705527},
MRREVIEWER = {John Hempel},
       DOI = {10.1112/blms/15.5.401},
       URL = {https://doi.org/10.1112/blms/15.5.401},
}

@article {MiguelManzano,
    AUTHOR = {Dom\'inguez-V\'azquez, Miguel and Manzano, Jos\'e{} M.},
     TITLE = {Isoparametric surfaces in {$\mathbb{E}(\kappa,\tau)$}-spaces},
   JOURNAL = {Ann. Sc. Norm. Super. Pisa Cl. Sci. (5)},
  FJOURNAL = {Annali della Scuola Normale Superiore di Pisa. Classe di
              Scienze. Serie V},
    VOLUME = {22},
      YEAR = {2021},
    NUMBER = {1},
     PAGES = {269--285},
      ISSN = {0391-173X,2036-2145},
   MRCLASS = {53A10 (53B25 53C30)},
  MRNUMBER = {4288655},
}

@article{MiguelTarciosTomas,
      title={Polar actions on homogeneous 3-spaces}, 
      author={Miguel Dominguez-Vazquez and Tarcios A. Ferreira and Tomas Otero},
      JOURNAL = {Annali di Matematica Pura ed Applicata},
      PAGES = {903--927},
      year={2026},
      VOLUME = {205},
      eprint={2505.05898},
      archivePrefix={arXiv},
      primaryClass={math.DG},
      DOI = {10.1007/s10231-025-01627-},
      url={https://doi.org/10.1007/s10231-025-01627-3}, 
}

@phdthesis{Filipkiewicz,
    title    = {Four-dimensional geometries},
    school   = {University of Warwick},
    author   = {Filipkiewicz, Richard },
    year     = {1983}, 
}

@Inbook{DiazRamos,
author="D{\'i}az-Ramos, Jos{\'e} Carlos
and Dom{\'i}nguez-V{\'a}zquez, Miguel
and Otero, Tom{\'a}s",

title="Homogeneous Hypersurfaces in Symmetric Spaces",
bookTitle="New Trends in Geometric Analysis: Spanish Network of Geometric Analysis 2007-2021",
year="2023",
publisher="Springer Nature Switzerland",
address="Cham",
pages="141--190",
isbn="978-3-031-39916-9",
doi="10.1007/978-3-031-39916-9_5",
url="https://doi.org/10.1007/978-3-031-39916-9_5"
}

@book{JSC,
    AUTHOR = {Berndt, J\"urgen and Console, Sergio and Olmos, Carlos
              Enrique},
     TITLE = {Submanifolds and holonomy},
    SERIES = {Monographs and Research Notes in Mathematics},
   EDITION = {Second},
 PUBLISHER = {CRC Press, Boca Raton, FL},
      YEAR = {2016},
     PAGES = {xxxviii+456},
      ISBN = {978-1-4822-4515-8},
   MRCLASS = {53C29 (53B25 53C40)},
  MRNUMBER = {3468790},
MRREVIEWER = {Miguel\ Dom\'inguez-V\'azquez},
       DOI = {10.1201/b19615},
       URL = {https://doi.org/10.1201/b19615},
}

@book{CeRyan,
    AUTHOR = {Cecil, Thomas E. and Ryan, Patrick J.},
     TITLE = {Geometry of hypersurfaces},
    SERIES = {Springer Monographs in Mathematics},
 PUBLISHER = {Springer, New York},
      YEAR = {2015},
     PAGES = {xi+596},
      ISBN = {978-1-4939-3245-0; 978-1-4939-3246-7},
   MRCLASS = {53C40 (53-02)},
  MRNUMBER = {3408101},
MRREVIEWER = {Juan\ de D. P\'erez},
       DOI = {10.1007/978-1-4939-3246-7},
       URL = {https://doi.org/10.1007/978-1-4939-3246-7},
}

@article{ManfioSantosVeken,
title = {Hypersurfaces of $\mathbb{S}^3\times \mathbb{R}$ and $\mathbb{H}^3\times \mathbb{R}$ with constant principal curvatures},
journal = {Journal of Geometry and Physics},
volume = {213},
pages = {105495},
year = {2025},
issn = {0393-0440},
doi = {https://doi.org/10.1016/j.geomphys.2025.105495},
url = {https://www.sciencedirect.com/science/article/pii/S0393044025000798},
author = {F. Manfio and J.B.M. {dos Santos} and J.P. {dos Santos} and J. {Van der Veken}}
}

@article{DIAZRAMOSComplexhyper,
title = {Isoparametric hypersurfaces in complex hyperbolic spaces},
journal = {Advances in Mathematics},
volume = {314},
pages = {756-805},
year = {2017},
issn = {0001-8708},
doi = {https://doi.org/10.1016/j.aim.2017.05.012},
url = {https://www.sciencedirect.com/science/article/pii/S0001870816303826},
author = {José Carlos Díaz-Ramos and Miguel Domínguez-Vázquez and Víctor Sanmartín-López}
}

@article{MiguelComplexPro,
 ISSN = {00029947, 10886850},
 URL = {https://www.jstor.org/stable/tranamermathsoci.368.2.1211},
 author = {Miguel Domínguez-Vázquez},
 journal = {Transactions of the American Mathematical Society},
 number = {2},
 pages = {1211--1249},
 publisher = {American Mathematical Society},
 title = {Isoparametric foliations on complex projective spaces},
 urldate = {2025-04-28},
 volume = {368},
 year = {2016}
}

@article{JoaoJoao,
title = {Isoparametric hypersurfaces in product spaces},
journal = {Differential Geometry and its Applications},
volume = {88},
pages = {102005},
year = {2023},
issn = {0926-2245},
doi = {https://doi.org/10.1016/j.difgeo.2023.102005},
url = {https://www.sciencedirect.com/science/article/pii/S0926224523000311},
author = {João Batista Marques {dos Santos} and João Paulo {dos Santos}}
}

@article{Urbano,
author = {Urbano, Francisco},
year = {2016},
month = {06},
pages = {},
title = {On hypersurfaces of $\mathbb{S}^2\times\mathbb{S}^2$},
volume = {27},
journal = {Communications in Analysis and Geometry},
doi = {10.4310/CAG.2019.v27.n6.a7}
}

@article{GaoMaYao,
    AUTHOR = {Gao, Dong and Ma, Hui and Yao, Zeke},
     TITLE = {On hypersurfaces of {$\mathbb{H}^2\times\mathbb{H}^2$}},
   JOURNAL = {Sci. China Math.},
  FJOURNAL = {Science China. Mathematics},
    VOLUME = {67},
      YEAR = {2024},
    NUMBER = {2},
     PAGES = {339--366},
      ISSN = {1674-7283,1869-1862},
   MRCLASS = {53C42 (53B25 53C40)},
  MRNUMBER = {4698771},
MRREVIEWER = {J.\ Carlos\ D\'iaz-Ramos},
       DOI = {10.1007/s11425-022-2103-2},
       URL = {https://doi.org/10.1007/s11425-022-2103-2},
}

@misc{dhaeneGuoxin,
      title={Homogeneous hypersurfaces of the four-dimensional Thurston geometry ${\rm Sol_0^4}$. arXiv 2508.10545}, 
      author={Marie D'haene and Guoxin Wei and Zeke Yao and Xi Zhang},
      year={2025},
      eprint={2508.10545},
      archivePrefix={arXiv},
      primaryClass={math.DG},
      url={https://arxiv.org/abs/2508.10545}, 
}

@article{Wall,
title = {Geometric structures on compact complex analytic surfaces},
journal = {Topology},
volume = {25},
number = {2},
pages = {119-153},
year = {1986},
issn = {0040-9383},
doi = {https://doi.org/10.1016/0040-9383(86)90035-2},
url = {https://www.sciencedirect.com/science/article/pii/0040938386900352},
author = {C.T.C. Wall}
}

@article{Thuong,
    AUTHOR = {Thuong, Scott Van},
     TITLE = {Metrics on 4-dimensional unimodular {L}ie groups},
   JOURNAL = {Ann. Global Anal. Geom.},
  FJOURNAL = {Annals of Global Analysis and Geometry},
    VOLUME = {51},
      YEAR = {2017},
    NUMBER = {2},
     PAGES = {109--128},
      ISSN = {0232-704X,1572-9060},
   MRCLASS = {53C30 (17B05 22E15)},
  MRNUMBER = {3612950},
MRREVIEWER = {Vladimir\ V.\ Gorbatsevich},
       DOI = {10.1007/s10455-016-9527-z},
       URL = {https://doi.org/10.1007/s10455-016-9527-z},
}

@article{Kodama,
author = {Kodama, H and Takahara, A and Tamaru, H},
year = {2011},
pages = {229 - 243},
title = {The space of left-invariant metrics on a Lie group up to isometry and scaling},
volume = {135},
journal = {manuscripta math},
doi = {10.1007/s00229-010-0419-4}
}

@article{Biggs,
    AUTHOR = {Biggs, Rory and Remsing, Claudiu C.},
     TITLE = {On the classification of real four-dimensional {L}ie groups},
   JOURNAL = {J. Lie Theory},
  FJOURNAL = {Journal of Lie Theory},
    VOLUME = {26},
      YEAR = {2016},
    NUMBER = {4},
     PAGES = {1001--1035},
      ISSN = {0949-5932},
   MRCLASS = {22E15 (17B05 22E60)},
  MRNUMBER = {3487553},
}

@article{Andrada,
author = {Andrada, Adrian and Barberis, Maria and Dotti, Isabel and Ovando, Gabriela},
year = {2004},
month = {03},
pages = {},
title = {Product structures on four dimensional solvable Lie algebras},
volume = {7},
journal = {Homology, Homotopy and Applications},
doi = {10.4310/HHA.2005.v7.n1.a2}
}

@article{Lee,
author = {Lee, Jong Bum and Lee, Kyung and Shin, J.-K and Yi, Seunghun},
year = {2007},
month = {09},
pages = {},
title = {Unimodular groups of type $\mathbb{R}^3 \rtimes \mathbb{R}$},
volume = {44},
journal = {Journal of the Korean Mathematical Society},
doi = {10.4134/JKMS.2007.44.5.1121}
}

@article {HaLee,
    AUTHOR = {Ha, Ku Yong and Lee, Jong Bum},
     TITLE = {The isometry groups of simply connected 3-dimensional
              unimodular {L}ie groups},
   JOURNAL = {J. Geom. Phys.},
  FJOURNAL = {Journal of Geometry and Physics},
    VOLUME = {62},
      YEAR = {2012},
    NUMBER = {2},
     PAGES = {189--203},
      ISSN = {0393-0440,1879-1662},
   MRCLASS = {53C30 (22E15)},
  MRNUMBER = {2864471},
MRREVIEWER = {Claudio\ Gorodski},
       DOI = {10.1016/j.geomphys.2011.10.011},
       URL = {https://doi.org/10.1016/j.geomphys.2011.10.011},
}

@article{Alekseevski,
author = {Alekseevskiĭ, D},
year = {2007},
month = {10},
pages = {87},
title = {Homogeneous Riemannian Spaces of Negative Curvature},
volume = {25},
journal = {Mathematics of the USSR-Sbornik},
doi = {10.1070/SM1975v025n01ABEH002200}
}

\end{document}